\documentclass[a4papaer,12pt]{article}
\usepackage{amsmath,amsthm,amssymb}
\usepackage{enumerate}
\usepackage{mathrsfs}
\usepackage{paralist}

\newtheorem{thm}{Theorem}[section]

\newtheorem{lem}[thm]{Lemma}
\newtheorem{prop}[thm]{Proposition}
\theoremstyle{dfn}
\newtheorem{defn}[thm]{Definition}

\theoremstyle{rem}

\numberwithin{equation}{section}

\newcommand{\e}{\varepsilon}

\DeclareMathOperator{\mis}{\mathfrak{M}}

\newcommand{\C}{\mathbb{C}}

\newcommand{\U}{\mathbf{1}}
\newcommand{\Sp}{\text{\textsc{sp}}}
\newcommand{\vsp}{\vec{\text{\textsc{sp}}}}

\newcommand{\pol}{\mathscr{P}}
\renewcommand{\c}{\mathscr{C}}

\newcommand{\bfa}{\mathbf{a}}

\newcommand{\set}[1]{\{#1\}}

\newcommand{\enorm}{\lVert\,\cdot\,\rVert}

\newcommand{\X}{{\triangle^{\!*}}}

\begin{document}

%-------------------------------------------------------------------------
% editorial commands: to be inserted by the editorial office
%
%\firstpage{1} \volume{228} \Copyrightyear{2004} \DOI{003-0001}
%\seriesextraline{This is the Concrete Title of this Book\br H.E. R and S.T.C. W, Eds.}
%\firstpage{1}

\title{A Characterization of Polynomially Convex Sets in Banach Spaces}

%----------Author 1
\author{Mortaza Abtahi\footnote{School of Mathematics and Computer Sciences,
Damghan University, Damghan, P.O.BOX 36715-364, Iran 
(\texttt{abtahi@du.ac.ir; mortaza.abtahi@gmail.com})}
$~$and Sara Farhangi}

% \thanks{This work was completed with the support of our \TeX-pert.}

%----------Author 2

%\subjclass{Primary 46E50, 46J10; Secondary 32E20, 46J20.}

%\keywords{Polynomials on Banach spaces, Polynomially convexity,
%Commutative Banach algebras, Joint spectrum.}

\date{\today}

\maketitle

\begin{abstract}
  Let $E$ be a Banach space and $\X$ be the closed unit ball of
  the dual space $E^*$.
  For a compact set $K$ in $E$, we prove that $K$ is polynomially convex
  in $E$ if and only if there exist a unital commutative Banach algebra $A$ and
  a continuous function $f:\X\to A$ such that
  \begin{inparaenum}[(i)]
    \item $A$ is generated by $f(\X)$,
    \item the character space of $A$ is homeomorphic to $K$, and
    \item $K=\vsp(f)$ the joint spectrum of $f$.
  \end{inparaenum}
  In case $E=\c(X)$, where $X$ is a compact Hausdorff space, we will see that
  $\X$ can be replaced by $X$.\footnote{2010 SMC: Primary 46E50, 46J10; Secondary 32E20, 46J20.\par
  }
\end{abstract}

\section{Introduction}
\label{sec:intro}

Let $K$ be a compact set in the $n$-dimensional complex space $\C^n$. It is said that
$K$ is \emph{polynomially convex} if for every $\lambda\in \C^n \setminus K$
there exists a polynomial $p$ such that $p(\lambda) = 1$ and $|p(w)|<1$ for all $w\in K$.
It is known that if $K\subset \C$ then $K$ is polynomially convex if and only if $\C\setminus K$
is connected (e.g. \cite[Theorem 2.3.7]{CBA}). A characterization of polynomially convex sets
in $\C^n$ is given as follows.
Suppose that $A$ is a commutative Banach algebra with unit element $\U$ and
character space $\mis(A)$. For an $n$-tuple $\bfa=(a_1,\dotsc,a_n)$ of elements
of $A$, the \emph{joint spectrum} $\vsp(\bfa)$, defined by \cite[Definition 2.3.2]{CBA}
\begin{equation}\label{eqn:joint-spectrum-finite}
  \vsp(\bfa)=\set{(\phi(a_1),\dotsc,\phi(a_n)):\phi\in\mis(A)},
\end{equation}
is a compact set in $\C^n$. In case $a_1,\dotsc,a_n$ generate $A$,
the space $\mis(A)$ is homeomorphic to $\vsp(\bfa)$ through the homeomorphism $\phi\mapsto (\phi(a_1),\dotsc,\phi(a_n))$.
The following theorem provides a characterization of those compact subsets of $\C^n$
which arise in this way as character spaces of commutative Banach algebras generated by $n$
elements.

\begin{thm}[e.g.\ {\cite[Theorem 2.3.6]{CBA}}]
\label{thm:K-in-Cn-is-pol-convex-iff}
  A compact set $K$ in $\C^n$ is polynomially convex if and only if
  there exists a unital commutative Banach algebra $A$ which is generated
  by $n$ elements $a_1,\dotsc,a_n$ with $K = \vsp(a_1,\dotsc,a_n)$.
\end{thm}

The main purpose of this paper is to establish analogous results for
compact sets in infinite dimensional Banach spaces. Throughout the paper,
the space of all continuous complex valued functions
on a compact Hausdorff space $X$ is denoted by $\c(X)$. For every $f\in\c(X)$,
the uniform norm of $f$ is defined by $\|f\|_X=\sup\set{|f(x)|:x\in X}$.
Given a Banach space $(E,\enorm)$,
the closed unit ball of the topological dual $E^*$ is denoted by $\X$,
which is a compact Hausdorff space for the weak* topology. For every
$x\in E$, the mapping $\tilde x:\X\to\C$, $\phi\mapsto \phi(x)$,
belongs to $\c(\X)$ with $\|\tilde x\|_\X=\|x\|$. Hence, the mapping
$x\mapsto \tilde x$ is an isometric isomorphism of $E$ onto a closed subspace of
$\c(\X)$. For this reason, we may consider $E$ as a closed subspace of $\c(\X)$.

The paper is outlined as follows. In Section \ref{sec:P(K)}, we briefly recall
some properties of polynomially convex sets in a Banach space $E$ and,
for a compact set $K$ in $E$, we study the algebra $\pol(K)$ of all
functions in $\c(K)$ that can be approximated uniformly on $K$ by polynomials.
The main results are included in Section \ref{sec:characterization}.
First, we briefly recall some properties of the joint spectrum $\vsp(f)$ of
an infinite system of elements $f:X\to A$ of a Banach algebra $A$. Next,
we consider the Banach space $\c(X)$ and present a characterization
of polynomially convex sets. Finally, we consider arbitrary Banach spaces
and prove that a compact set $K$ in a Banach space $E$ is polynomially convex
if and only if there exist a unital commutative Banach algebra $A$ and
a continuous function $f:\X\to A$ such that
\begin{inparaenum}[(i)]
  \item $A$ is generated by $f(\X)$,
  \item the character space of $A$ is homeomorphic to $K$, and
  \item $K=\vsp(f)$.
\end{inparaenum}

\section{Algebras of Polynomials on Banach Spaces}
\label{sec:P(K)}

Let $E$ and $F$ be complex Banach spaces. A mapping $P:E\to F$ is called
an \emph{$n$-homogeneous polynomial} if there exists a symmetric continuous
$n$-linear map $A:E^n\to F$ such that $P(x)=A(x,\dotsc,x)$, for every
$x\in E$. Any finite sum of homogeneous polynomials is called
a \emph{polynomial} from $E$ into $F$.

The simplest and, perhaps, the most useful polynomials are obtained by multiplying
linear functionals together. For a finite set $\{\phi_1,\dotsc,\phi_n\}$
of functionals in $E^*$ and a vector $y\in F$, the mapping $P:x\mapsto \phi_1(x)\dotsm \phi_n(x)y$
defines an \emph{$n$-homogeneous polynomial of finite type} from $E$ into $F$.
By the polarization formula \cite[Theorem 1.10]{Mujica-CA-in-BS}, every $n$-homogeneous
polynomial of finite type can be expressed in the form
\begin{equation}
  P(x)= \phi_1^n(x)y_1+\phi_2^n(x)y_2+\dotsb+\phi_m^n(x)y_m,\ x\in E,
\end{equation}
where $\phi_1,\dotsc,\phi_m \in E^*$ and $y_1,\dotsc,y_m \in F$. Any
finite sum of homogeneous polynomials of finite type is called
a \emph{polynomial of finite type}. Since, in this paper, we
consider only polynomials of finite type, we reserve the notation $\pol(E,F)$ for
the space of such polynomials from $E$ into $F$, and
we write $\pol(E)$ for $\pol(E,\C)$. For more information on polynomials
in Banach spaces see, for example, \cite{Dineen-CA-on-IDS},
\cite{Mujica-CA-in-BS}, or \cite{Prolla-1977}.

We \emph{declare} that proofs of some results presented in this section may be
found in the literature. The proofs, however, are included for the reader's convenience.

\begin{defn}
\label{dfn:poly-convex-hull-of-K}
  Let $K$ be a compact set in $E$. Denoted by $\hat K_E$,
  the \emph{polynomially convex hull} of $K$ in $E$ is defined by
  \begin{equation}\label{eqn:poly-convex-hull-of-K}
    \hat K_E = \{x\in E: |P(x)|\leq \|P\|_K,\ P\in \pol(E)\},
  \end{equation}

  \noindent
  where $\enorm_K$ is the uniform norm over $K$.
  The set $K$ is \emph{polynomially convex} if $K=\hat K_E$.
\end{defn}

We will often write $\hat K$ instead of $\hat K_E$ when the underlying
  space $E$ is tacitly understood.
First of all, let us show that polynomially convexity is not altered when passing
from a subspace to the space.

\begin{prop}
\label{prop:KE=KF}
  Let $E$ be a closed subspace of a Banach space $F$, and $K$ be
  a compact set in $E$. Then the polynomially convex hull of $K$
  in $E$ and the polynomially convex hull of $K$ in $F$ coincide;
  that is, $\hat K_E=\hat K_F$.
\end{prop}

\begin{proof}
  Since the restriction to $E$ of every polynomial in $\pol(F)$ is a polynomial
  in $\pol(E)$, we have $\hat K_E\subset \hat K_F$. To prove the reverse inclusion,
  take an element $x\in \hat  K_F$. First, we show that $x\in E$. If $\phi\in F^*$
  and $\phi|_E=0$ then $\|\phi\|_K=0$. Regarding $\phi$ as a polynomial in $\pol(F)$,
  we get $\phi(x)=0$. Since $E$ is closed in $F$,
  the Hahn-Banach Theorem implies that $x\in E$. Another use of the Hahn-Banach
  Theorem shows that every polynomial $P\in \pol(E)$ extends to a polynomial
  $\bar P\in \pol(F)$. Clearly, $\|P\|_K=\|\bar P\|_K$ and thus,
  for $x\in \hat K_F$,
  \[
    |P(x)| = |\bar P(x)| \leq \|\bar P\|_K = \|P\|_K.
  \]
  Since $x\in E$, we obtain $x\in \hat K_E$. Therefore, $\hat K_F\subset \hat K_E$.
\end{proof}

We remark that, by \cite[Corollary 2.4]{Mujica-Oka-Weil}, if $K$ is compact
then $\hat K$ is also compact.

\begin{defn}
  Let $K$ be a compact set in a Banach space $E$. We define $\pol_0(K)$ to be
  the space of the restriction to $K$ of all polynomials $P\in\pol(E)$. We define
  $\pol(K)$ as the closure of $\pol_0(K)$ in $\c(K)$.
\end{defn}

Clearly, $\pol(K)$ is a uniform algebra on $K$.

\begin{prop}\label{prop:P(K)=P(hat K)}
  The algebra $\pol(K)$ is isometrically isomorphic to $\pol(\hat K)$.
\end{prop}

\begin{proof} Clear.
% Define a mapping $\eta: \pol(\hat{K}) \to \pol(K)$ by $\eta (f)=f|_K$.
% Since $\|P\|_K=\|P\|_{\hat K}$, for every $P\in\pol(E)$,
% we get $\|f\|_K=\|f\|_{\hat K}$, for all $f\in\pol(\hat K)$,
% and thus $\eta$ is an isometry. To see that $\eta$ is surjective,
% take a function $g\in \pol(K)$. There
% exists a sequence $\{P_n\}$ of polynomials in $\pol(E)$ such that $P_n\to g$ uniformly
% on $K$. Since $\|P_n-P_m\|_{\hat K}=\|P_n-P_m\|_K$, the sequence $\{P_n\}$ is
% Cauchy in $\c(\hat K)$. Hence, $P_n\to f$, for some $f\in \c(\hat K)$.
% Clearly, $f|_K=g$ and thus $\eta$ is surjective.
\end{proof}

Next, we identify the character space of $\pol(K)$ as the polynomially convex hull
$\hat K$. Before that, the following lemma is in order.

\begin{lem}\label{lem:w*-conv-uni-conv}
  Let $\{\psi_\alpha:\alpha\in I\}$ be a bounded net in $E^*$ that
  converges, in the weak* topology, to some $\psi\in E^*$.
  Then $\psi_\alpha \to \psi$, uniformly on compact sets in $E$.
\end{lem}

\begin{proof}
  Assume that $\{\psi_\alpha\}$ is bounded by $M$ and that
  $K$ is a compact set in $E$. Given $\e>0$,
  there exist  $x_1,\ldots, x_n \in K$ such that $K\subset \bigcup_{i=1}^n B(x_i,\e)$,
  where $B(x,\e)$ is the open ball with center $x$ and radius $\e$.
  Set
  \[
    U_0=\{\psi \in E^* : |\psi(x_i)-\psi_0(x_i)|<\e, i=1,\dotsc,n\}.
  \]

  Then $U_0$ is a neighbourhood of $\psi_0$ in the weak* topology of $E^*$.
  There is $\alpha_0\in I$ such that $\psi_\alpha \in U_0$, for $\alpha\geq \alpha_0$.
  For every $x\in K$, there exists $i\in \{1,\ldots,n\}$ such that $\|x- x_i\|<\e$.
  Now, for $\alpha\geq \alpha_0$,
  \begin{align*}
     |\psi_\alpha(x)-\psi_0(x)|
       & \leq |\psi_\alpha(x)- \psi_\alpha(x_i)|+|\psi_\alpha(x_i)- \psi_0(x_i)|
         +|\psi_0(x_i)-\psi_0(x)| \\
       & \leq M\e +\e+\|\psi_0\|\e \leq (2M+1)\e.
  \end{align*}

  Since $x\in K$ is arbitrary, we get $\|\psi_\alpha-\psi_0\|_K \leq (2M+1)\e$,
  for $\alpha\geq \alpha_0$. This shows that $\psi_\alpha\to \psi_0$ uniformly on $K$.
\end{proof}

\begin{thm}\label{thm:m(P(K))=hat K}
 The character space of $\pol(K)$ is homeomorphic to $\hat K$.
\end{thm}

\begin{proof}
  Using Proposition \ref{prop:P(K)=P(hat K)}, we assume that
  $K$ is polynomially convex. Define a mapping $J:K \to \mis(\pol(K))$ by
  $J(x)=\delta_x$, where $\delta_x:f\mapsto f(x)$ is the evaluation homomorphism
  of $\pol(K)$. The mapping $J$ imbeds $K$ homeomorphically as a compact subset of
  $\mis(\pol(K))$; see \cite[Chapter 4]{Dales}. We just need to show that $J$
  is surjective. Take  a character $\phi$ of $\pol(K)$. Since $E^*\subset \pol(E)$,
  we may consider $\phi$ as a linear functional on $E^*$. We show that $\phi$ is weak*
  continuous on $E^*$. Using a result of Banach (e.g.~\cite[Corollary 4, p 250]{Horvath}),
  it suffices to show that $\phi$ is weak* continuous on bounded subsets of $E^*$.
  Suppose that $\{\psi_\alpha\}$ is a bounded net in $E^*$ converging to $\psi_0$
  in the weak* topology. By Lemma \ref{lem:w*-conv-uni-conv},
  $\psi_\alpha\to \psi_0$ uniformly on $K$ which, in turn, implies that
  $\phi(\psi_\alpha)\to \phi(\psi_0)$. Hence, $\phi$ is a weak* continuous functional
  on $E^*$. By \cite[Theorem 3.10]{Rudin-FA}, there is $x\in E$ such that
  $\phi(\psi)=\psi(x)$ for all $\psi\in E^*$.
  Therefore, $\phi(P)=P(x)$ for every polynomial $P\in \pol(E)$.
  Moreover, since $\phi$ is a character of $\pol(K)$, we have $|P(x)|=|\phi(P)| \leq \|P\|_K$,
  for every polynomial $P\in \pol(E)$. This shows that $x\in \hat K=K$.
  Since $\phi=\delta_x$ on $\pol(E)$, a dense subset of $\pol(K)$,
  we get $\phi=\delta_x$ on $\pol(K)$.
\end{proof}

\section{A Characterization of Polynomially Convex Sets}
\label{sec:characterization}

In this section, we present the main results of the paper
characterizing polynomially convex compact sets in Banach spaces.
To start we briefly recall some properties of the joint spectrum of an
infinite system of elements of a Banach algebra.

Suppose that $A$ is a unital commutative Banach algebra, that $X$ is a nonempty set,
and that $f:X\to A$ is a function. The \emph{joint spectrum} of $f$ is defined to be the set of
all functions $\lambda:X\to \C$ such that the ideal of $A$ generated by
$\set{\lambda(x)\U-f(x):x\in X}$ is proper. To distinguish between the usual spectrum
and the joint spectrum, the former is denoted by $\Sp(f)$ and
the latter is denoted by $\vsp(f)$.
It is proved (e.g.\ see \cite{ab-fa}) that
\begin{equation}\label{eqn:sp(f)=phi o f}
  \vsp(f)=\set{\phi\circ f : \phi \in \mis(A)}.
\end{equation}

It is worth mentioning that when $X$ is enriched with some structure (topological, algebraical, etc.),
and $f:X\to A$ is an appropriate morphism, many structural properties of $f$ are inherited
by every $\lambda\in \vsp(f)$. For more on joint spectrum, see
\cite{Dineen-Harte-Taylor-I} and \cite{Harte-gmj}, for example.

%Now, we consider the special case where $E=\c(X)$, with $X$
%a compact Hausdorff space, and start with the following result.

\begin{prop}
\label{prop:if-A=<f(X)>-then-mis(A)=SP(f)}
  Suppose that $A$ is a unital commutative Banach algebra, that $X$ is
  a compact Hausdorff space, and that $f:X\to A$ is a continuous function.
  If $f(X)$ generates $A$, then the mapping $\Phi: \mis(A) \to \vsp(f)$,
  $\phi \mapsto \phi \circ f$, is a homeomorphism.
\end{prop}

\begin{proof}
  Since $f:X\to A$ is continuous, the joint spectrum $\vsp(f)$
  is a compact set in $\c(X)$; see e.g.\ \cite[Theorem 3.3]{ab-fa}.
  We just need to prove that $\Phi$ is a continuous bijection.
  Obviously $\Phi$ is surjective. We show that $\Phi$ is injective.
  Assume that $\phi_1\circ f=\phi_2\circ f$, for $\phi_1,\phi_2\in \mis(A)$.
  Then $\phi_1=\phi_2$ on the range $f(X)$.
  Since $f(X)$ generates $A$, we get $\phi_1=\phi_2$ and thus $\Phi$ is injective.

  The continuity of $\Phi$ follows from \cite[Proposition 3.5]{Abtahi-vector-valued}.
  It is proved that, with respect to the weak* topology of $A^*$, the mapping
  $A^*\to \c(X)$, $\phi\mapsto\phi\circ f$, is continuous on bounded subsets of $A^*$.
  Since $\mis(A)$ is bounded, we conclude that $\Phi$ is continuous.
\end{proof}

%\begin{rem}
%  It is proved, in \cite[Lemma 2.3.3]{CBA}, that if $S$ is a subset of $A$
%  that generates $A$ then $\mis(A)$ is homeomorphic to a subset of
%  $\prod_{a\in S} \Sp(a)$. This subset is, in fact, $\vsp(f)$
%  where $f:S\to A$ is the inclusion map.
%\end{rem}

Before presenting our main results, the following lemma is in order.

\begin{lem}
\label{lem:P(f)-is-dense-in-A}
  Let $A$ be a unital commutative Banach algebra, $X$ be a compact Hausdorff space,
  and $f:X\to A$ be a continuous function.
  \begin{enumerate}[\upshape(i)]
    \item \label{item:Pf-is-defined}
    With each polynomial $P\in\pol(\c(X))$ is associated an element $Pf$ of $A$
    such that $\phi(Pf)=P(\phi\circ f)$, for all $\phi\in \mis(A)$.
    \item \label{item:Pf=<f(X)>}
    The set $\set{Pf:P\in\pol(\c(X))}$ is an algebra by itself and
    is dense in the subalgebra of $A$ generated by $f(X)$.
  \end{enumerate}

\end{lem}

\begin{proof}
  \eqref{item:Pf-is-defined}
  We start with the simplest polynomials in $\pol(\c(X))$,
  that is functionals $\psi\in \c(X)^*$. By the Riesz Representation Theorem,
  there exists a regular complex Borel measure $\mu$ such that
  \[
    \psi(g) = \int_X g d\mu, \quad (g\in \c(X)).
  \]

  By \cite[Theorem 3.27]{Rudin-FA}, the vector-valued integral
  $\int fd\mu$ is defined as an element of $A$ such that
  $\phi(\int fd\mu)=\int\phi\circ f d\mu$, for all $\phi\in A^*$.
  Define $\psi f=\int fd\mu$ and $\psi^mf=(\psi f)^m$, for $m\geq0$.
  Since each polynomial $P\in\pol(\c(X))$ is a linear combination of
  $\set{\psi^m:\psi\in \c(X)^*, m\geq 0}$, we see that $Pf$,
  as an element of $A$, is defined with no ambiguity.
  Moreover, $\phi(Pf)=P(\phi\circ f)$, for all $\phi\in\mis(A)$.
  Note, however, that the latter equality may fail to hold for
  $\phi\in A^*$.

  \eqref{item:Pf=<f(X)>} Without loss of generality, assume that $A$ is generated by $f(X)$.
  Since $\pol(\c(X))$ is an algebra, the set $A_0=\set{Pf:P\in\pol(\c(X))}$ is an algebra.
  Since $f(x)=\delta_x f$ and $\delta_x\in \pol(\c(X))$, for every $x\in X$,
  we get $f(X) \subset A_0$. Therefore, $A_0$ is dense in $A$.
\end{proof}

Extending Theorem \ref{thm:K-in-Cn-is-pol-convex-iff}, we now give a
characterization of polynomially convex sets in $\c(X)$, where $X$ is
a compact Hausdorff space.

\begin{thm}\label{thm:K-subset-C(X)}
  For a compact set $K$ in $\c(X)$, the following statements are equivalent.
  \begin{enumerate}[\upshape(i)]
    \item \label{item:K-is-pol-convex-in-C(X)}
    The set $K$ is polynomially convex.

    \item \label{item:exist-A-and-f}
    There is a unital commutative Banach algebra $A$ and a continuous function $f:X\to A$
    such that $f(X)$ generates $A$ and $K=\vsp(f)$.
  \end{enumerate}
\end{thm}

\begin{proof}
  $\eqref{item:K-is-pol-convex-in-C(X)} \Rightarrow \eqref{item:exist-A-and-f}$
  Assume that $K$ is polynomially convex.
  For every $x\in X$, consider the evaluation functional $\delta_x:g\mapsto g(x)$
  of $\c(X)$. Define a function $f:X\to \c(K)$ by $f(x)=\delta_x$.
  We show that $f$ is continuous. %Using the Arzel\'a-Acoli Theorem,
  Since $K$ is compact, as a subset of $\c(X)$,
  it is equicontinuous. This means that, given $\e>0$,
  every $x_0\in X$ has a neighborhood $U_{x_0}$ such that
  \[
    |g(x)-g(x_0)| \leq \e, \quad (x\in U_{x_0},g\in K).
  \]
  This implies that $\|f(x)-f(x_0)\|_K < \e$ for $x\in U_{x_0}$.
  Therefore, $f$ is continuous.

  Take $A$ as the closed subalgebra of $\c(K)$ generated by $f(X)$.
  We claim that $A=\pol(K)$. Since every element of $f(X)$ is a continuous functional
  of $\c(X)$, we have $f(X)\subset \pol(K)$ and thus $A\subset\pol(K)$.
  To see the reverse inclusion,
  we just need to show that the restriction to $K$ of
  every polynomial $P\in \pol(\c(X))$ belongs to $A$. Since every such polynomial
  is a finite linear combination of $\set{\phi^m:\phi\in \c(X)^*,m\geq 0}$, it suffices to
  show that the restriction to $K$ of every functional $\phi\in \c(X)^*$ belongs to $A$.
  This can be seen by the Krein-Millman Theorem. It is known that the set
  of extreme points of the closed unit ball of $\c(X)^*$ is exactly the
  set $f(X)=\set{\delta_x:x\in X}$ of unit point masses. Hence, by the Krein-Millman Theorem,
  every functional $\phi\in \c(X)^*$ is a weak* limit of finite linear combinations
  of elements of $f(X)$. By Lemma \ref{lem:w*-conv-uni-conv}, convergence with respect
  to the weak* topology yields uniform convergence on compact sets. Hence,
  every functional $\phi\in \c(X)^*$ is a uniform limit, on $K$, of elements
  of $A$. We conclude that $\pol(K)\subset A$.

  By Theorem \ref{thm:m(P(K))=hat K}, since $K$ is polynomially convex, we get
  $\mis(A)=K$. Also, for every $g\in K$, we have $\phi_g\circ f=g$, for
  \[
    (\phi_g\circ f)(x) = \phi_g(f(x)) = \phi_g(\delta_x) = \delta_x(g)=g(x),
    \quad(x\in X).
  \]
  Therefore,
  \[
    \vsp(f) = \set{\phi \circ f: \phi\in \mis(A)}
    = \set{\phi_g\circ f: g \in K}
    = \set{g : g\in K} = K.
  \]

  $\eqref{item:exist-A-and-f} \Rightarrow \eqref{item:K-is-pol-convex-in-C(X)}$
  Assume that $K=\vsp(f)$, where $f:X\to A$ is a continuous function with
  $A$ a unital commutative Banach algebra generated by $f(X)$. Take a function
  $g\in \hat K$. To verify that $g\in K$, we show that $g=\phi\circ f$,
  for some character $\phi\in\mis(A)$.
  Let $A_0=\set{Pf:P\in\pol(\c(X))}$ be the subalgebra of
  $A$ defined in Lemma \ref{lem:P(f)-is-dense-in-A}.
  Define a mapping $\phi_g:A_0\to\C$ by $\phi_g(Pf)=Pg$.
  We show that $\phi_g$ is a well-defined continuous homomorphism.
  Since $g\in \hat K$, we have, for every $P\in \pol(\c(X))$,
  \begin{equation}\label{eqn:Pf-Pg}
  \begin{split}
    |\phi_g(Pf)| & = |Pg| \leq \|P\|_K \\
     & = \sup\{|P(\phi\circ f)|: \phi\in \mis(A))\} \\
     & = \sup\{|\phi(P f)|: \phi\in \mis(A)\} \leq \|P f\|.
  \end{split}
  \end{equation}

  Therefore, $Pf=0$ implies $Pg=0$ which shows that $\phi_g$ is well-defined.
  Obviously, $\phi_g$ is a homomorphism. Also \eqref{eqn:Pf-Pg} shows that
  $\phi_g$ is continuous on $A_0$. Since $A$ is generated by $f(X)$,
  by Lemma \ref{lem:P(f)-is-dense-in-A}, $A_0$ is dense in $A$. Hence, $\phi_g$ extends
  to a character of $A$, still denoted by $\phi_g$. We have
  \begin{equation*}
    (\phi_g\circ f)(x) = \phi_g(f(x)) = \phi_g(\delta_x f)=\delta_x(g) =g(x), \quad (x\in X).
  \end{equation*}
  Therefore, $g=\phi_g\circ f$ and thus $g\in K$.
\end{proof}

We conclude the paper with our final result on characterising polynomially convex
sets in arbitrary Banach spaces.

\begin{thm}\label{thm:main}
  For a compact set $K$ in a Banach space $E$, the following are equivalent.
  \begin{enumerate}[\upshape(i)]
    \item The set $K$ is polynomially convex in $E$.
    \item There is a unital commutative Banach algebra $A$
    and a continuous function $f:\X\to A$ such that $A$ is
    generated by $f(\X)$ and $K=\vsp(f)$. Moreover, $K$
    is homeomorphic to $\mis(A)$.
  \end{enumerate}
\end{thm}

\begin{proof}
  It follows from the fact that $E$ is isometrically isomorphic
  to a closed subspace of $\c(\X)$, along with
  propositions \ref{prop:KE=KF} and \ref{prop:if-A=<f(X)>-then-mis(A)=SP(f)},
  and Theorem \ref{thm:K-subset-C(X)}.
\end{proof}

Aa an illustration of Theorem \ref{thm:main}, the following example is given.

\subsubsection*{Example}
  Let $E$ be an arbitrary Banach space and let $K=\set{e_0,e_1,e_2,\dotsc}$ be a
  set in $E$ with $e_n\to e_0$. We show that $K$ is polynomially convex. Take
  $A=\c(K)$, which is, in fact, the space of all convergent complex sequences.
  Define a function $f:\X\to A$ by $f(\psi)(x)=\psi(x)$, $x\in K$.
  Since $K$ is compact, by Lemma \ref{lem:w*-conv-uni-conv}, $f$ is continuous.
  Let $s_j=(\delta_{ij})_{i=1}^\infty$, where $\delta_{ij}$ is the Kronecker delta.
  Then $S=\{s_1,s_2,\dotsc\}$ generates $A$ and $S\subset f(\X)$.
  Hence, $f(\X)$ generates $A$. Finally, since $\mis(A)=\set{\delta_{x}:x\in K}$
  and
  \[
    \delta_{x}\circ f(\psi)=\delta_{x}(f(\psi))=f(\psi)(x)=\psi(x)=\delta_x(\psi) \quad (x\in K)
  \]
  we get
  \[
    \vsp(f)=\set{\delta_x\circ f: x\in K} = \set{\delta_x:x\in K} = K.
  \]

  To be more precise, note that the last equality is, in fact, a homeomorphism.
  All requirements in Theorem \ref{thm:main} are fulfilled and $K$ is polynomially convex.

\subsection*{Acknowledgment}
The authors express their sincere gratitude to the referee for his/her careful
reading and suggestions that improved the presentation of this paper.

\end{document}